\documentclass[a4paper,reqno,12pt]{amsart}

\usepackage{amsmath, amssymb, amsthm, enumerate,amsfonts,latexsym, mathrsfs}

\usepackage[top=30mm,right=27mm,bottom=30mm,left=27mm]{geometry}

\newtheorem{theorem}{Theorem}[section]
\newtheorem{corollary}[theorem]{Corollary}
\newtheorem{lemma}[theorem]{Lemma}
\newtheorem{proposition}[theorem]{Proposition}

\newtheorem{remark}[theorem]{Remark}

\makeatletter
\newcommand{\imod}[1]{\allowbreak\mkern4mu({\operator@font mod}\,\,#1)}
\makeatother

\newtheorem*{thmA}{Theorem~A}
\newtheorem*{thmB}{Theorem~B}
\newtheorem*{thmD}{Theorem~D}
\newtheorem*{thmC}{Theorem~C}
\newtheorem*{thmE}{Theorem~E}

\newcommand{\Irr}{{\mathrm {Irr}}}
\newcommand{\cd}{{\mathrm {cd}}}
\newcommand{\cod}{{\mathrm {cod}}}
\newcommand{\cdc}{{\mathrm {cdc}}}
\newcommand{\dl}{{\mathrm {dl}}}
\newcommand{\cv}{{\mathrm {cv}}}
\newcommand{\ncv}{{\mathrm {ncv}}}
\newcommand{\Aut}{{\mathrm {Aut}}}

\newcommand{\PSL}{{\mathrm {PSL}}}

\newcommand{\PSU}{{\mathrm {PSU}}}

\newcommand{\Atlas}{{\sf Atlas}}
\newcommand{\GAP}{{\sf GAP}}

\theoremstyle{definition}

\newtheorem{problem}[theorem]{Problem}

\begin{document}
\title[\textbf{Finite groups with few character values}]{\textbf{Finite groups with few character values that are not character degrees}}

\author[S. Y. Madanha]{Sesuai Y. Madanha}
\address{Department of Mathematics and Applied Mathematics, University of Pretoria, Private Bag X20, Hatfield, Pretoria 0028, South Africa}
\email{sesuai.madanha@up.ac.za}

\author[X. Mbaale]{Xavier Mbaale}
\address{Department of Mathematics and Statistics, University of Zambia, Lusaka, Zambia}
\email{xavier@aims.ac.za}

\author[T. M. Mudziiri Shumba]{Tendai M. Mudziiri Shumba}
\address{Sobolev Institute of Mathematics, Novosibirsk, Russia}
\email{tendshumba@gmail.com}

\subjclass[2010]{Primary 20C15}

\date{\today}

\keywords{character values, character degrees, derived length, solvable groups, non-solvable groups}

\begin{abstract}
Let $ G $ be a finite group and $ \chi \in \Irr(G) $. Define $ \cv(G)=\{\chi(g)\mid \chi \in \Irr(G), g\in G \} $, $ \cv(\chi)=\{\chi(g)\mid g\in G \} $ and denote $ \dl(G) $ by the derived length of $ G $. In the 1990s Berkovich, Chillag and Zhmud described groups $ G $ in which $ |\cv(\chi)|=3 $ for every non-linear $ \chi \in \Irr(G) $ and their results show that $ G $ is solvable. They also considered groups in which $ |\cv(\chi)|=4 $ for some non-linear $ \chi \in \Irr(G) $. Continuing with their work, in this article, we prove that if $ |\cv(\chi)|\leqslant 4 $ for every non-linear $ \chi \in \Irr(G) $, then $ G $ is solvable. We also considered groups $ G $ such that $ |\cv(G)\setminus \cd(G)|=2 $. T. Sakurai classified these groups in the case  when $ |\cd(G)|=2 $. We show that $ G $ is solvable and we classify groups $ G $ when $ |\cd(G)|\leqslant 4 $ or $ \dl(G)\leqslant 3 $. It is interesting to note that these groups are such that $ |\cv(\chi)|\leqslant 4 $ for all $ \chi \in \Irr(G) $. Lastly, we consider finite groups $ G $ with $ |\cv(G)\setminus \cd(G)|=3 $. For nilpotent groups, we obtain a characterization which is also connected to the work of Berkovich, Chillag and Zhmud. For non-nilpotent groups, we obtain the structure of $ G $ when $ \dl(G)=2 $.
\end{abstract}

\maketitle


\section{Introduction}\label{s:intro}

Many results have been obtained using the character degrees set of a finite group. Recently, restricting the study to the real valued or rational valued characters of finite groups to determine the structure of the group has received a lot of attention. This has proved that character values in a character table encode a great deal of information about the group structure of a finite group. In this article, our first result is based on a problem first considered by Berkovich, Chillag and Zhmud in the 1990s in \cite{BCZ95}. 

Let $ G $ be a finite group and $ \Irr(G) $ be the set of complex irreducible characters of $ G $. Recall the definition of a character degrees set and that of character values set of $ G $:

\begin{center}
$ \cd(G):=\{\chi(1)\mid \chi \in \Irr(G) \} $,
$ \cv(G)=\{\chi(g)\mid \chi \in \Irr(G), g\in G \} $,
\end{center} respectively. If $ \chi \in \Irr(G) $, then denote $ \cv(\chi)=\{\chi(g)\mid g\in G \} $. It follows that $ \cv(\chi)\subseteq \cv(G) $ for all $ \chi \in \Irr(G) $. Note also that for every non-linear $ \chi \in \Irr(G) $, $ |\cv(\chi)|\geqslant 3 $.

In \cite{BCZ95}, the authors studied groups $ G $ with the following extremal property: $ |\cv(\chi)|=3 $ for every non-linear $ \chi \in \Irr(G) $.  Their results show that $ G $ is a $ 2 $-group or some solvable Frobenius group \cite[Main Theorem]{BCZ95}. The authors also considered groups in which $ |\cv(\chi)|=4 $ for some non-linear $ \chi \in \Irr(G) $. Our first result is a continuation of their investigation and we show that if $ |\cv(\chi)|\leqslant 4 $ for every non-linear $ \chi \in \Irr(G) $, then $ G $ is solvable.

\begin{thmA}
Let $ G $ be a finite group. If $ |\cv(\chi)|\leqslant 4 $ for all non-linear $ \chi \in \Irr(G) $, then $ G $ is solvable.
\end{thmA}

Our proof uses the classification of finite simple groups, the existence of extendible rational valued characters in finite simple groups and the fact that non-solvable groups all of whose non-linear irreducible characters are rational valued are rational groups. At the end of Section 5, we give some possibilities and some examples of groups that satisfy the hypothesis of Theorem A. We also note that $ A_{5} $ has a character with five character values and so the condition in Theorem A is best possible.

For the rest of the article, we shall study a subset of $ \cv(G) $, the complement of $ \cd(G) $. Define $ \mathrm{cdc}(G):=\cv(G)\setminus \cd(G) $. Our objective is to study how the structure of $ G $ is affected by  $ |\cdc(G)| $. It is easy to see that $ \cdc(G) =\emptyset $ if and only if $ G=1 $ and that $ |\cdc(G)|=1 $ if and only if $G \cong C_{2}^{*} $, where $ C_{n}^{*} $ denotes a direct product of cyclic groups of order $ n $. It is also easy to see that if $ G $ is non-abelian, then $ |\cdc(G)|\geqslant 2 $. What if $ |\cdc(G)|=2 $? A moment's thought would have us conclude that $ G $ is a rational group. Rational groups have been studied extensively(see \cite{Row} and \cite{DGD24} and their references for more details). In \cite{Sak21}, Sakurai considered this problem in the special case when $ |\cd(G)|=2 $. He showed that $ |\cd(G)|=2 $ if and only if $ G\cong C_{3}^{*}\rtimes C_{2} $, a Frobenius group with an elementary abelian $ 3 $-group kernel and whose complement is of order two. We show that in the general case, $ G $ is solvable. In fact, our result shows that when $ |\cdc(G)|=3 $, with a few possible exceptions, $ G $ is solvable:

\begin{thmB}
Let $ G $ be a finite group whose composition factors are not isomorphic to either $ A_{5} $ or $ A_{6} $. If $ |\cdc(G)|\leqslant 3 $, then $ G $ is solvable.
\end{thmB}

We believe that the condition on composition factors is not necessary. In order to prove the above theorem, we work with a subset of $ \cdc(G) $ that allows us to have an inductive proof. In particular, we focus on elements $ a $ of $ \cdc(G) $ such that $ a\notin \mathbb{N} $. The drawback of this strategy is for example: For any normal subgroup $ N $ of $ G $, $ 2\in \cdc(G/N) $ if $ G/N\cong A_{5} $ but it is not clear if $ 2\in \cdc(G) $. This necessitates the hypothesis on $ A_{5} $ and $ A_{6} $. Note that non-solvable groups with $ |\cdc(G)|=4 $ and $ |\cd(G)|=4 $ were classified in \cite{Mad22}.

As an extension of \cite[Theorem]{Sak21}, we consider finite groups in which $ |\cd(G)|\leqslant 4 $ or $ \dl(G)\leqslant3 $ in the following result:

\begin{thmC}
Let $ G $ be a finite group and suppose that either $ |\cd(G)|\leqslant 4 $ or $ \dl(G)\leqslant3 $. Then $ |\cdc(G)|=2 $ if and only if one of the following holds:
\begin{itemize}
\item[(a)] $ \cd(G)=\{1,2\} $ and $ G $ is a Frobenius group with a Frobenius kernel which is an elementary abelian $ 3 $-group and whose complement is of order $ 2 $.
\item[(b)] $ |\cd(G)|=\{1, 2, 3\} $ and $ G\cong S_{4} $.
\end{itemize}
\end{thmC}
 
The groups considered in Theorem C are necessarily solvable from the hypothesis since in the first case $ |\cv(G)| = 6 < 8 $(see \cite[Theorem 1.1]{Mad22}) and the second case is obvious. We also note that these groups are such that $ |\cv(\chi)|\leqslant 4 $ for all $ \chi \in \Irr(G) $ and hence are connected to groups in Theorem A. Based on calculations in \GAP{} \cite{GAP}, we think groups in Theorem C are the only finite groups such that $ |\cdc(G)|=2 $. However, we cannot show this at this point. We consider a counterexample in the last section. In particular, we show that the group necessarily has an element of order $ 6 $ and a character degree divisible by $ 6 $.

In our next result, we characterise nilpotent groups such that $ |\cdc(G)|=3 $. Note that the groups described in Theorem D are some of the groups studied by Berkovich, Chillag and Zhmud in \cite[Corollary]{BCZ95}. The result is easy to prove and we put it here because of its interesting connection to the work of Berkovich, Chillag and Zhmud.  We recall the definition of a codegree of an irreducible character: If $ \chi \in \Irr(G)$, then $ \cod(\chi)=|G{:}\ker \chi|/\chi(1) $.

\begin{thmD}
Let $ G $ be a finite nilpotent group. Then the following statements are equivalent:
\begin{itemize}
\item[(a)] $ |\cdc(G)|=3 $;
\item[(b)] $ |\cd(G)|=2 $ and $ |\cv(\chi)|\leqslant 3 $ for all $ \chi \in \Irr(G) $;
\item[(c)] $ |\cd(G)|=2 $, $ G $ is a $ 2 $-group with $ \cod(\chi)=2\chi(1) $ for all non-linear $ \chi \in \Irr(G) $.
\end{itemize}
If one of the above statements hold, then there exists a normal subgroup $ N $ such that $ G/N $ is an extraspecial $ 2 $-group.
\end{thmD}

 Classifying the groups above seems to be too technical. Indeed the $ 2 $-groups with five character values in their character table are too many as remarked in \cite{Sak21}. 
 However, these $ 2 $-groups have a factor group which is extraspecial as mentioned in Theorem D.

For the non-nilpotent groups, we investigate groups $ G $ when $ |\cdc(G)|=3 $ and $ \dl(G)=2 $. We have that Sylow $ 2 $-subgroups of these groups are either abelian or precisely the $ 2 $-groups with five character values that are described in Theorem E.

\begin{thmE}
Let $ G $ be a finite non-nilpotent group. Suppose that $ |\cdc(G)|=3 $. If $ \dl(G)=2 $, then $ G=(C_{3}^{*}\times O_{2}(G))\rtimes C_{2} $ with $ O_{2}(G) $ abelian. Moreover, if $ P $ is a non-abelian Sylow $ 2 $-subgroup of $ G $, then $ |\cv(P)|=5 $ and $ G/P'\cong (C_{3}^{*}\rtimes C_{2}) \times C_{2}^{*} $, where $ C_{3}^{*}\rtimes C_{2} $ is a Frobenius group whose kernel is an elementary abelian $ 3 $-group whose complement is of order $ 2 $.
\end{thmE}

Our paper is organized as follows: Section \ref{prelim} has some preliminary results; in Section \ref{cdc2} we consider groups such that $ |\cdc(G)|=2 $ and either $ |\cd(G)|\leqslant 4 $ or $ \dl(G)\leqslant 3 $ and we establish Theorem C. In Section \ref{cdc3} we prove Theorems D and E. In Section \ref{nonsolvable}, we prove Theorems A and B. In the last section we list some open problems that arise in the article. Our notation is standard and we follow \cite{Isa06} for the character theory of finite groups.

\section{Preliminary Results}\label{prelim}
In this section we list some preliminary results that we need to prove our main results. We shall use freely these well known results in the following lemma. 

Define $ \ncv(G):=\{\chi(g)\notin \mathbb{N} : \chi \in \Irr(G), g\in G \} $ and  $ \ncv(\chi):=\{\chi(g)\notin \mathbb{N} : g\in G \} $ for a fixed $ \chi \in \Irr(G) $. It follows that $ \ncv(G) \subseteq \cdc(G) $ and we shall indicate that $a\in \ncv(G)$ where it is necessary to make sure that $ a\in \cdc(G) $.
\begin{lemma}\label{lem}
Let $ G $ be a finite group, $ \chi \in \Irr(G) $, $ N $ be a normal subgroup of $ G $ and $ n $ be a positive integer. Then the following holds:
\begin{itemize}
\item[(a)] $ |\cv(G/N)|\leqslant |\cv(G)| $ and $ |\ncv(G/N)|\leqslant |\ncv(G)| $;
\item[(b)] The character table of $ C_{n} $ has a column with pairwise different values. In particular, $ |\cv(C_{n})|=n $;
\item[(c)] If $ G $ is non-abelian, then $ 0\in \ncv(G) $;
\item[(d)] If $ \cv(\chi)\subseteq \mathbb{Q} $, then there exists a negative rational number $ a\in \ncv(\chi) $;
\item[(e)] If $ G=H\times K $ and $ n\in \cv(H) $ and $ m\in \cv(K) $, then $ n, m, mn\in \cv(G) $;
\item[(f)] If $ \chi(g)=a $ for some $ g\in G $, where $ a \notin \mathbb{Q} $, then there exists $ h\in G $ such that $ \chi(h)=b $, with $ b\notin \mathbb{Q} $ and $ a\not=b $.
\end{itemize}
\end{lemma}
\begin{proof}
(a) - (c) follows from \cite[Lemma 2.1]{Mad22}. (d) follows from the first orthogonality relation \cite[Corollary 2.14]{Isa06} using $ \chi $ and the principal character. (e) follows from the fact that for every $ \chi \in \Irr(G) $, $ \chi = \theta\times \varphi $ for some $ \theta \in \Irr(H) $ and $ \varphi \in \Irr(K) $.

For (f), let $ \epsilon $ be a $ |G| $-th root of unity and $ \sigma \in \mathrm{Gal}(\mathbb{Q}(\epsilon)/\mathbb{Q}) $. If $ \chi(g)=a $, then $ (\chi(g))^{\sigma}=\chi(h^{v}) \in \cv(\chi) $, where $ h\in G $, $ v\in \mathbb{Z} $, $ \gcd(v,|G|)=1 $ and $ v $ is such that $ \epsilon^{\sigma}=\epsilon^{v} $. If $  (\chi(g))^{\sigma}=\chi(g) $, then $ a $ is rational, a contradiction and so $  (\chi(g))^{\sigma}\not=\chi(g) $ and irrational.
\end{proof}



\begin{lemma} \label{charactercenter}
Let $ G $ be a finite group, $ \chi \in \Irr(G) $ and $ Z=Z(\chi) $. Then 
\begin{itemize}
\item[(a)] $ \chi _{Z}=\chi(1) \lambda $ for some linear character $ \lambda $ of $ Z $;
\item[(b)] $ Z/\ker \chi = Z(G/\ker \chi) $ is cyclic;
\item[(c)] $ \chi (g^{-1})=\overline{\chi(g)} $.
\end{itemize}
\end{lemma}
\begin{proof}
These follow from \cite[Lemmas 2.15 and 2.27]{Isa06}.
\end{proof}

%

\begin{lemma} \cite[Lemma 1.1]{QSY04}\label{prootofunityelement}
Let $ G $ be a finite group, $ g\in G $ and $ p $ be a prime. If $ g $ is a $ p $-element and $ |\chi(g)|^{2}=a $ for some rational number $ a $, then $ p\mid \chi(1)^{2} - a $. In particular, if $ \chi(g)=0 $, then $ p\mid \chi(1) $ and if $ |\chi(g)|=1 $, then $ \gcd(\chi(1), ord(g))=1 $.
\end{lemma}

The following lemma has some basic results on character values of nilpotent groups. 
\begin{lemma}\label{nilpotentvalues}
Let $ G $ be a finite non-abelian nilpotent group and $ \chi \in \Irr(G) $ be non-linear. Then the following holds: 
\begin{itemize}
\item[(a)] $ \chi(1)\epsilon,\chi(1)\overline{\epsilon}  \in \ncv(G) $, for some $ 1\not= \epsilon $, a root of unity.
\item[(b)]$ |\ncv(G)|\geqslant 3 $. Moreover, 
\begin{itemize}
\item[(i)] If $ G $ is not a $ 2 $-group, then $ |\ncv(G)|\geqslant 5 $.
\item[(ii)] If $ |\cd(G)|\geqslant 3 $, then $ |\cdc(G)|\geqslant 4 $.
\end{itemize}
\item[(c)] If $ G $ is a $ p $-group, then $ |\chi(g)|\not=1 $ for all $ g\in G $.
\end{itemize}
\end{lemma}
\begin{proof}
(a) Let $ \chi \in \Irr(G) $ be non-linear. Then $ \overline{G}=G/\ker \chi $ is non-abelian and $ 1\not= Z(\overline{G}) $ is cyclic by Lemma \ref{charactercenter}(b). Let $ \overline{z}\in Z(\overline{G}) $ be non-trivial. Then by Lemma \ref{charactercenter}(a), $ \chi(\overline{z})=\chi(1)\lambda(\overline{z})=\chi(1)\epsilon $ for some non-principal linear character $ \lambda $ of $ Z(\overline{G}) $ and $ \epsilon \not=1 $, a root of unity. It follows that $ \chi(1)\overline{\epsilon} $ using Lemma \ref{charactercenter}(c).

(b) Since $ G $ is non-abelian, $ |\cv(G)|\geqslant 4 $. If $ |\cv(G)|= 4 $, then $ \cv(G)=\{1, -1, 0, b(G)\} $ and $ |\ncv(G)|=2 $ by \cite[Lemma 4.1]{Sak21}. But from (a), $ b(G)\epsilon \in \ncv(G) $, for some root of unity $ \epsilon\not= 1 $, a contradiction.

For (i), suppose that $ G $ is not a $ 2 $-group. Let $ G $ be of even order and let $ G=P\times H $, where $ P $ is the Sylow $ 2 $-subgroup of $ G $ and $ H $ is the Hall $ 2' $-subgroup of $ G $. Suppose that $ P $ is non-abelian. Then $ |\ncv(P)| \geqslant 3 $ and $ |\ncv(H)|\geqslant 2 $ and so $ \ncv(P)\cup \ncv(H)\subseteq \ncv(G)$ by Lemma \ref{lem}(e). It also follows that $|\ncv(G)|\geqslant 5 $ by Lemma \ref{lem}(e). We may assume that $ H $ is non-abelian. It is sufficient to consider a non-abelian $ 3 $-group. Then $ \epsilon, \overline{\epsilon}, 1, m, 0 \in \cv(G) $, where $ m=\chi(1) $ for some non-linear $ \chi \in \Irr(G) $ and $ \epsilon  \not= 1$ is a root of unity. By (a), $ m\epsilon, m\overline{\epsilon}  \in \ncv(G) $ with $ m\epsilon\not= m\overline{\epsilon} $. The result then follows.

(ii) follows from the fact that for every character degree $ m $ of $ G $, there is $ m\epsilon\in \ncv(G) $, where $ \epsilon  \not= 1$ is a root of unity. 

(c) This follows from Lemma \ref{prootofunityelement}.
\end{proof}

We now turn to root of unity elements. Recall that an element $ g \in G $ is called a root of unity element if $ |\chi(g)|=1 $ for all $ \chi \in \Irr(G) $. Groups with this element were characterised by Lewis, Morotti and Tong-Viet in \cite{LMT23}, continuing the work of Ostrovskaya and Zhmud'(see \cite{LMT23} for reference). We first define some notation in the result: 
\textit{Let $ q > 2 $. Denote $ \Gamma_{q} $ to be the unique doubly transitive Frobenius group with a cyclic complement of order $ q-1 $ and degree $ q $.}
\begin{lemma}\cite[Theorem B]{LMT23}\label{rootofunity}
Let $ G $ be a finite group. Then $ G $ has a root of unity element if and only if one of the following holds:
\begin{itemize}
\item[(a)] $ G $ is abelian;
\item[(b)] $ F(G)=G'Z(G) $ is abelian, $ G'\cap Z(G)= 1 $; and $ G/Z(G)\cong \Gamma_{q_{1}} \times \cdots \times \Gamma_{q_{m}} $, where each $ q_{i} > 2 $ is a prime power and $ m\geqslant 1 $ is an integer.
\end{itemize}
\end{lemma}

\begin{lemma} \cite[Theorem 12.5]{Isa06} If $ \cd(G)=\{1, m\} $, then one of the following holds:
\begin{itemize}
\item[(a)] $ G $ has an abelian normal subgroup of index $ m $;
\item[(b)] $ m=p^e $ for a prime $ p $ and $ G $ is the direct product of a $ p $-group and abelian group.
\end{itemize}
\end{lemma}

%
%

The following lemma will come in handy in the next section.
\begin{lemma}\label{coprimerootofunity}
Let $ G=N\rtimes P $ where $ N $ is a minimal normal subgroup of $ G $, $ P $ is an abelian $ p $-subgroup with $ O_{p}(G)=1 $ and $ \gcd(|N|, |P|)=1 $. Then the following holds;
\begin{itemize}
\item[(a)] If $ 1\not= n \in N $, then $ n\notin \ker \chi $ for all non-linear $ \chi \in \Irr(G) $.
\item[(b)] If $ \chi \in \Irr(G) $ is non-linear and $ n\in Z(\chi)\cap N $, then $ \chi(1) \epsilon \in \ncv(G)\setminus \cv(P) $ for some $ \epsilon \not= 1 $, a root of unity. If $ |N| $ is odd, then $ \chi(1) \epsilon\not= \chi(1)\overline{\epsilon} \in \ncv(G)\setminus \cv(P) $. 
\item[(c)] If $ n\notin Z(\chi) $ for all non-linear $ \chi \in \Irr(G) $, then $ G\cong \Gamma_{q} $, $ q $ a prime power order.
\end{itemize}
\end{lemma}
\begin{proof}
(a) If $ n\in \ker \chi $, then $ G/\ker \chi $ is abelian which means that $ \chi $ is linear.

(b) By Lemma \ref{charactercenter}(a), $ \chi_{Z(\chi)}=\chi(1)\lambda $ for some linear character $ \lambda \in \Irr(Z(\chi)) $. Since $ n\notin \ker \chi $, $ \chi(n) =\chi(1)\lambda(n) $ with $ \lambda \not= 1 $, otherwise $ n\in \ker \chi $. If $ |N| $ is odd, then $ \epsilon\not= \pm 1 $ and so $ \epsilon \not= \overline{\epsilon} $. The result then follows.

(c) It follows that $ n $ is a root of unity element and since $ Z(G)=1 $, the result follows by Lemma \ref{rootofunity}. 
\end{proof}


We need the following result to prove Theorem A:

\begin{lemma}\cite[Propositions 1 and 2 and Theorem 8]{SS98} \label{3conjugacyclasses}
Let $ G $ be a finite group and $ N $ be a normal subgroup of $ G $ which is a union of three conjugacy classes in $ G $. Then one of the following holds: 
\begin{itemize}
\item[(a)] $ N $ is an elementary abelian $ p $-group of odd order; 
\item[(b)] $ N $ is a metabelian $ p $-group;
\item[(c)] $ N $ is a Frobenius group with a cyclic complement of prime order.
\end{itemize} 
\end{lemma}
\section{Finite groups such that $ |\cdc(G)|=2 $}\label{cdc2}

\subsection{Groups with less than or equal to three character degrees.}

The objective of this section is to generalize T. Sakurai's theorem. Our proof exhibits the abundance of root of unity elements in some factor groups of groups with four character values. 

\begin{lemma} \cite[Lemma 4.2]{Sak21}\label{4charactervalueslemma}
Let $ G $ be a finite non-abelian group such that $ |\cdc(G)|=2 $ with $ \cd(G)=\{1, m\} $. Then the following holds: 
\item[(i)] $ G/G' $ is a non-trivial elementary abelian $ 2 $-group;
\item[(ii)] $ |C_{G}(t)|= |G:G'| $ for every involution $ t $ of $ G $;
\item[(iii)] $ m=2^{n} $ for some positive integer $ n $.
\end{lemma}

\begin{theorem}\label{4charactervalues}
Let $ G $ be a finite group such that $ \dl(G) =2 $. Then $ |\cdc(G)|=2 $ if and only if $ G $ is a Frobenius group with a Frobenius kernel which is an elementary abelian $ 3 $-group and complement of order $ 2 $. In this case $ |\cd(G)|=2 $.
\end{theorem}

\begin{proof}
If $ G $ is a Frobenius group with a Frobenius kernel which is an elementary abelian $ 3 $-group and whose complement is of order $ 2 $, then $ |\cd(G)|=2 $ and $ |\cdc(G)|=2 $. 

Suppose that $ \dl(G)=2 $ and $ |\cdc(G)|)= 2 $. Note that $ G/G' $ is a non-trivial elementary abelian $ 2 $-group by \cite[Proposition 2.3]{Sak21}, $ G' $ is abelian and $ 0, -1\in \cdc(G) $. Let $ G'=O_{2}(G)\times H $, where $ H $ is the abelian Hall $ 2' $-subgroup of $ G $. Let $ P $ be a Sylow $ 2 $-subgroup of $ G $. Then $ G/H\cong P $. If $ P $ is non-abelian, then $ 3\leqslant |\ncv(P)|\leqslant |\cdc(G)| $ by Lemma \ref{nilpotentvalues}(b). We may assume that $ P $ is elementary abelian and so $ G'=H $. If $ O_{2}(G)\not= 1 $ and $ P=P_{1}\times O_{2}(G) $, then $ G=O_{2}(G)\times (H\rtimes P_{1}) $. Using Lemma \ref{lem}(e), we have that $ -\chi(1)\in \ncv(G) $, for some non-linear $ \chi \in \Irr(G) $, a contradiction. Hence we may assume that $ O_{2}(G)= 1 $. Note that $ G' $ is a direct product of abelian Sylow subgroups. Suppose that $ q $ divides $ |G'| $, where $ q $ is an odd prime. Let $ M $ be a normal subgroup of $ G $ such that $ G'/M=\overline{G}' $ is a minimal normal subgroup of $ G/M=\overline{G} $ and $ q $ divides $ |\overline{G}'| $. Suppose that $ \overline{g}\in \overline{G}' $ is a $ q $-element. Then $ \overline{g}\notin \ker \chi $ for every non-linear $ \chi \in \Irr(\overline{G}) $ and every non-linear irreducible character of $ \overline{G} $ is faithful. Note that every non-linear character of $ \overline{G} $ is of $ 2 $-power degree. By \cite[Theorem 3.13]{Isa06}, all these characters are of the same degree. In particular, $ \cd(\overline{G})=\{ 1, m\} $ and so $ \cv(\overline{G})=\{ 1, -1, 0, m\} $. By Lemma \ref{prootofunityelement}, $ \chi(\overline{g})\in \{ 1, -1\} $ for all $ \chi \in \Irr(\overline{G}) $, that is, $ \overline{g} $ is a root of unity element of $ \overline{G} $. Using Lemma \ref{rootofunity}, we have that $ \overline{G}/Z(\overline{G})\cong \Gamma_{q_{1}} \times \cdots \times \Gamma_{q_{m}} $, where each $ q_{i} > 2 $ is a prime power and $ m\geqslant 1 $ is an integer. Hence $ q_{1}q_{2}\cdots q_{m}=3 $ and so $ \overline{G}/Z(\overline{G})\cong \mathrm{S}_{3} $, which means that $ \overline{g} $ is a $ 3 $-element. This means that $ |P|=2 $. Hence $ \cv(G)=\{1, -1, 0, 2\}  $.  Note that $ \overline{G}'\cap Z(\overline{G})=1 $ and so $ Z(\overline{G})=1 $ using Lemma \ref{4charactervalueslemma}(ii). Since $ M $ is arbitrary, $ G' $ is a $ 3 $-group. Let $ t $ be an involution of $ G $. If $ C_{G}(t) $ is not a $ 2 $-group, then there is a $ 3 $-element $ g\in C_{G}(t) $ which means that $ \langle g \rangle\in Z(G) $ and so $ G=Z(G)\times K $ for some subgroup $ K $ of $ G $, a contradiction. We may assume that $ C_{G}(t) $ is a $ 2 $-group and so $ t $ acts fixed point freely on $ G' $. From \cite[Proposition 16.9(e)]{Hup98}, it follows that every element of $ G' $ is inverted by $ t $. Let $ g\in G' $ be of maximal order. If $ |\langle g \rangle |=k $, $ k\geqslant 9 $, then there exists a normal subgroup $ N $ of $ G $ such that $ G/N $ is a non-abelian group of order $ 18 $. Hence $ G/N \cong D_{9}$ and so $ |\ncv(G)|\geqslant|\ncv(G/N)| \geqslant 5 $. Therefore $ G' $ is an elementary abelian $ 3 $-group. By Lemma \ref{4charactervalueslemma}(ii), $ C_{G}(t)\subseteq P $ and so $ G $ is a Frobenius group with Frobenius kernel $ G' $ and complement of order $ 2 $. The result follows. 
\end{proof}

\begin{theorem}\label{3,5} Let $ G $ be a finite non-abelian group such that $ |\cd(G)|=3 $. Then $ |\cdc(G)|=2 $ if and only if $ G\cong S_{4} $.
\end{theorem}
\begin{proof}
Suppose that $ \cd(G)=\{1, m, n\} $. It follows that $ G/G' $ is an elementary abelian $ 2 $-group and so $ \cv(G)=\{1, m, n, 0, -1\} $. Note also that $ \dl(G)\leqslant 3 $, $ G'' $ is abelian and $ G/G'' $ is metabelian. By Lemma \ref{nilpotentvalues}(b), $ G $ is not nilpotent. We have that for all $\chi\in \Irr(G/G'') $, $ \chi(1) $ divides $ |G:G'| $ by Ito's theorem \cite[Theorem 6.15]{Isa06}. This means that $ m $ is a power of $ 2 $. Let $ P $ be a Sylow $ 2 $-subgroup.

Suppose first that $ \gcd(m, n)\not= 1 $. This means $ 2 $ divides $ n $. By Thompson's theorem \cite[Corollary 12.2]{Isa06}, $ G $ has a normal nilpotent $ 2 $-complement $ N $. If $ P  $ is non-abelian, then $ -m\in \ncv(P)\subseteq \ncv(G) $ by Lemma \ref{nilpotentvalues}(a), a contradiction.
 Hence $ P $ is elementary abelian. We have that $ \dl(G/G'')=2 $ and $ |\cdc(G/G'')|=2 $ and so by Theorem \ref{4charactervalues}, $ G/G'' $ is a Frobenius group whose kernel is an elementary abelian $ 3 $-group and has a complement of order $ 2 $. In particular, $ |P|=2 $. Since every non-linear irreducible character is divisible by $ 2 $, it is of $ 2 $-defect zero. If $ t \in P $ is an involution, then every non-linear $ \chi \in \Irr(G) $ vanishes on $ t $. This means that $ C_{G}(t)\subseteq P $. Hence $ G $ is a Frobenius group with a complement of order $ 2 $ by \cite[Theorem ]{Isa08}. Since a complement of $ G $ is of even order, we have that the Frobenius kernel is abelian. In this case, $ |\cd(G)|=2 $, a contradiction.  
 
We may assume that $ \gcd(m, n)=1 $. Then $ \cdc(G/G'')=\{0, -1\} $. By Sakurai's theorem \cite[Theorem]{Sak21}, $ G/G''\cong C_{3}^{*}\rtimes C_{2} $ and so $ G/G' $ is cyclic of order $ 2 $. Then $ P/O_{2}(G) $ is cyclic of order $ 2 $. Note that $ n $ divides $ |G:G''| $ by Ito's theorem and so $ n $ is a power of $ 3 $. Let $ Q $ be a Sylow $ 3 $-subgroup of $ G $. By \cite[Theorem A]{IMNT09}, $ Q/O_{3}(G)\cong Q/Q' $ is cyclic of order $ 3 $. Consider the abelian group $ T=O_{3}(G)O_{3'}(G) $. Then $ G/G'\cong C_{2} $, $ G'/T\cong C_{3} $ and $ G/T\cong S_{3} $. This means that $ T=G'' $ and $ n=3 $. Let $ \theta \in \Irr(G') $ be non-linear. Let $ \chi \in \Irr(G) $ be an irreducible constituent of $ \theta ^{G} $. Then $ \chi _{G'} \in \Irr(G') $ since $ \gcd(\chi(1), |G:G'|)=1 $. Hence for every non-linear $ \theta \in \Irr(G') $, there exists $ \chi \in \Irr(G) $ such that $ \chi_{G'}=\theta $. This means that any character value of $ G' $ from a non-linear character is also a character value of $ G $. Suppose that $ Q $ is non-abelian. Then $ G'/O_{3'}(G)\cong Q $, a $ 3 $-group. By Lemma \ref{nilpotentvalues},(a) $ 3\epsilon \in \cdc(G) $ for some cube root of unity $ \epsilon \not= 1 $, a contradiction. Therefore $ Q $ is abelian. If $ Q $ is not elementary abelian, then $ G/O_{3'}(G)\cong Q\rtimes C_{2} $ such that $ |\cv(Q\rtimes C_{2})|> 4 $ by Sakurai's theorem \cite[Theorem]{Sak21}. Note that $ \cd(Q\rtimes C_{2})=\{1,2\} $ and so any other possible character values of $ Q\rtimes C_{2} $ are not positive integers. Hence $ \cdc(G)\geqslant 3 $, a contradiction. Hence $ Q $ is elementary abelian and $ G'=O_{3}(G)\times (O_{3'}(G)\rtimes C_{3}) $. This means that $ 3\epsilon \in \ncv(G') \subseteq \cdc(G) $ for some cube root of unity $ \epsilon \not= 1 $, a contradiction. Hence $ G'=O_{3'}(G)\rtimes C_{3} $. Note that $ G' $ has exactly three linear characters. Consider $ 1\not= c\in C_{3} $. Every non-linear irreducible character of $ G' $ vanishes on $ c $. Then $ C_{G'}(c)=\sum _{\theta \in \Irr(G')}|\theta(c)|^{2}=1^{2} + 1^{2} + 1^{2}= 3 $. This means that $ C_{G'}(c)\subseteq C_{3} $ for all $ 1\not= c\in C_{3} $, that is, $ G' $ is a Frobenius group with a complement of order $ 3 $. Note that $ \cd(G)=\{ 1, 2, 3\} $ and $ G=(O_{3'}(G)\rtimes C_{3})\rtimes C_{2} $. By \cite[Theorem 31.3(2)]{Hup98}, $ A=O_{3'}(G) $ is an elementary abelian $ 2 $-group. In particular, $ |O_{3'}(G)|=4 $. Therefore $ G\cong S_{4} $ as required.
\end{proof}

\subsection{Groups whose derived length is three} 
\begin{theorem}\label{derivedlength3}
Let $ G $ be a finite group such that $ \dl(G) =3 $. Then $ |\cdc(G)|=2 $ if and only if $ G\cong S_{4} $.
\end{theorem}
\begin{proof}
Suppose that $ G $ is not isomorphic to $ S_{4} $. First note that $ 0, -1\in \cdc(G) $. Let us first consider $ G/G'' $. Then $ \dl(G/G'')=2 $ and $ |\cdc(G)|=2 $. By Theorem \ref{4charactervalues}, $ G/G''\cong C_{3}^{*}\rtimes C_{2} $. This means that $ 2\in \cd(G) $. We also have that for all non-linear $ \chi \in \Irr(G) $, $ \chi(1) $ divides $ |G:G''| $, that is, $ \chi(1) $ divides $ 2\cdot 3^{r} $ for some integer $ r $. 

If $ 2 $ divides all $ \chi(1) $ for all non-linear $ \chi\in \Irr(G) $, then by Thompson's theorem \cite[Corollary 12.2]{Isa06}, $ G $ has a normal $ 2 $-complement $ H $. Let $ P $ be a Sylow $ 2 $-subgroup of $ G $. Then $ G/H\cong P $. If $ P $ is non-abelian, then $ -2\in \cdc(G) $, a contradiction. Hence $ P $ is elementary abelian. But $ G/G' $ is of order $ 2 $ and so $ |P|=2 $. This means that every non-linear irreducible character is of $ 2 $-defect zero. Let $ t $ be an involution in $ G $. Then every non-linear $ \chi \in \Irr(G) $ vanishes on $ t $. Thus $ C_{G}(t)\subseteq P $ and so $ G $ is a Frobenius group with a Frobenius complement of order 2. This means that $ G $ has an abelian kernel, a contradiction since $ \dl(G)=3 $. 

We may assume that there exists $ \chi\in \Irr(G) $ such that $ \chi(1)=3^{r} $ for some integer $ r $. It follows that $ \chi_{G'}\in \Irr(G') $. Let $ Q $ be a Sylow $ 3 $-group of $ G $. If $ Q $ is non-abelian, then since $ G' $ has a factor group isomorphic to $ Q $, we have a contradiction by Lemma \ref{nilpotentvalues}(a). Hence $ Q $ is abelian. Note that $ G'=O_{3'}(G)\rtimes Q $. If $ Z(G')\not= 1 $, then $ G'=Z(G')\times S $ for some subgroup $ S $ of $ G' $ and so $ \chi=\lambda \times \theta $, where $ \lambda \in \Irr(Z(G')) $ and $ \theta\in \Irr(S) $. Hence there exists $ g\in G $ such that $ \chi(g) \in \ncv(G') \subseteq \ncv(G) $, a contradiction. We may assume that  $ Z(G)=1 $. By \cite[Theorem 3.13]{Isa06}, $ 3^{r} $ is the only character degree of $ G '$ which is a $ 3 $-power and $ |G'/G''|=3^r $. Hence if $ 1\not= s\in Q $, then $ C_{G'}(s) $ is a $ 3 $-group and it follows that $ G' $ is a Frobenius group with a complement $ Q $. Since $ Q $ is abelian and of odd order, $ Q $ is cyclic of order $ 3 $ and $ r=1 $. If $ \cd(G)=\{1, 2, 3\} $, then by Theorem \ref{3,5}, $ G\cong S_{4} $. Hence we may assume that there exists another character degree of $ G $ divisible by $ 6 $ and so the character degree is $ 6 $. 
 Since $ G $ is a rational group, by \cite[Corollary 2]{Row}, $ K $ is a $ \{2, 5\} $-group. Now using \cite[Theorem A]{DGD24}, we have that $ K $ is a $ 2 $-group. Hence $ \cd(G)=\{1,2,3,6\} $ and $ G $ has normal subgroups $ M $ and $ N $ such that $ G/G'' \cong S_{3} $, $ G'/G''\cong C_{3} $ and $ G' $ is a Frobenius group with Frobenius complement $ C_{3} $ and an abelian kernel $ G'' $ which is a $ 2 $-group.
 
We claim that $ G $ has no element of order $ 6 $. Suppose there is an element $ h $ of order $ 6 $ in $ G $. Consider $ G' $ which is a Frobenius group with a complement $ C_{3} $. Then $ C_{G'}(c)\subseteq C_{3} $ for all $ 1\not=c\in C_{3} $ and so $ \langle h \rangle \cap G''=1 $. Hence $ G=G''\langle h \rangle $ and so $ G/G'' \cong \langle h \rangle $ contradicting the fact that $ G/G'' \cong S_{3} $. Therefore $ G $ contains $ 2 $-elements and $ 3 $-elements only. Consider $ \chi \in \Irr(G) $ such that $ \chi(1)=6 $. Note that $ \chi $ vanishes on $ 3 $-elements. Consider a $ 2 $-element $ g $ in $ G'\setminus \ker \chi $. Then $ \chi(g)\not= \pm 1 $ by Lemma \ref{prootofunityelement}. Therefore $ -1\notin \cv(\chi) $, contradicting Lemma \ref{lem}(d). Hence $ 6\notin \cd(G) $. The result then follows.
\end{proof}

\begin{theorem}\label{4characterdegrees}
Let $ G $ be a solvable group with $ |\cd(G)|=4 $. Then $ |\cdc(G)|\geqslant 3 $. 
\end{theorem}
\begin{proof} Suppose that $ |\cdc(G)|= 2 $. Note that $ \dl(G)\leqslant 4 $ by \cite[Theorem C]{IK98}. It is sufficient to show that $ \dl(G)\leqslant 3 $.
Using Theorems \ref{4charactervalues} and \ref{derivedlength3}, we may assume that $ \dl(G)=4 $. Firstly, let us consider $ G/G''' $. Note that $ |\cdc(G/G''')|=2 $. By Theorem \ref{derivedlength3}, $ G/G'''\cong S_{4} $. If $ \cd(G)=\{1, 2, 3, 6\} $, then using \cite[Theorem 1.1]{QS04}, we have that $ \dl(G)\leqslant 3 $. 

It follows that $ \cd(G)=\{ 1, 2, 3, l\} $, where $ 6\mid l $ and $ l\not= 6 $. Note that $ G $ has Fitting height $ 3 $ by \cite[Theorem A]{Rie00} and $ F_{2}/F $ is cyclic, where $ F=F(G) $ and $ F_{2} $ is such that $ F_{2}=F(G/F) $. Considering the argument in \cite{DL15}, $ G $ is the group investigated in \cite[Section 6]{DL15} and hence is the group described in \cite[Theorem 1.1(2)]{DL15}. In particular, $ \cd(G)=\{1, 2, 3, 3.2^{a}\} $ where $ a\geqslant 2 $, $ |G/F_{2}|=2 $, $ |F_{2}/F|=3 $ and there is a normal subgroup $ N\leqslant F $ and positive integer $ b $ such that $ F/N $ is elementary abelian of order $ 2^{b} $ and $ (2^{2b}-1)/(2^{b}-1) $ divides $ 3 $ and $ 3 $ divides $ 2^{2b}-1 $. If $ b\geqslant 2 $, then we have a contradiction. We may assume that $ b=1 $. It follows that $ G/N $ is of order $ 12 $ and has a normal subgroup of order $ 2 $, $ F/N $. This means that $ F/N $ is central and hence $ G/N\cong D_{12} $, the dihedral group of order $ 12 $, a contradiction.
\end{proof}
Combining Theorems \ref{4charactervalues} - \ref{4characterdegrees}, we conclude that Theorem C holds. As mentioned in the introduction, we investigate a counterexample of the general case in the last section.

\section{Finite solvable groups such that $ |\cdc(G)|=3 $}\label{cdc3}
In this section, we consider groups $ G $ such that $ |\cdc(G)|=3 $. We start with an easy proof of a result on nilpotent groups. 
\begin{theorem}\label{nilpotent3values}
Let $ G $ be a finite non-abelian nilpotent group. Then $ |\cdc(G)|=3 $ if and only if $ G $ is $ 2 $-group with $ |\cd(G)|=2 $ and $ |\cv(\chi)|\leqslant 3 $ for all $ \chi \in \Irr(G) $.
\end{theorem}
\begin{proof}
If $ |\cd(G)|\geqslant 3 $, then $ |\cdc(G)|\geqslant 4 $ by Lemma \ref{nilpotentvalues}(b)(ii). Hence $ |\cd(G)|=2 $. Using Lemma \ref{nilpotentvalues}(b)(i), it follows that $ G $ is a $ 2 $-group. This means that $ |\cv(G)|= 5 $. In particular, $ \cv(G)=\{1, m, -1, -m, 0 \} $ by Lemma 2.4(a). Note that $ 1, -1\in \cv(\chi) $ for all non-linear $ \chi \in \Irr(G) $ and hence $ |\cv(\chi)|=3 $ for all non-linear $ \chi \in \Irr(G) $. It also follows that $ |\cv(\chi)|=2 $ for all linear $ \chi \in \Irr(G) $. The result then follows.

The other direction also follows by using Lemma \ref{nilpotentvalues}.
\end{proof}
Theorem \ref{nilpotent3values} together with \cite[Main Theorem]{BCZ95} gives us the first part of Theorem D. The second part of Theorem D is going to be proved in Theorem \ref{3valuesdl2}(a).

\begin{theorem} \label{3valuesdl2}Let $ G $ be a finite group with $ |\cdc(G)|=3 $. If $ \dl(G)=2 $, then one of the following holds:

\begin{itemize}
\item[(a)] $ G $ is a $ 2 $-group and there exists a normal subgroup $ N $ such that $ G/N $ is extraspecial.
\item[(b)]  $ G=(C_{3}^{*}\times O_{2}(G))\rtimes C_{2} $ with $ O_{2}(G) $ is abelian. Moreover, if $ P $ is a non-abelian Sylow $ 2 $-subgroup of $ G $, then $ \cv(P)=\{1,2,0,-1,-2 \} $ and $ G/P'\cong (C_{3}^{*}\rtimes C_{2}) \times C_{2}^{*} $, where $ C_{3}^{*}\rtimes C_{2} $ is a Frobenius group whose kernel is an elementary abelian $ 3 $-group and a complement of order $ 2 $.
\end{itemize}

\end{theorem}
\begin{proof}
Suppose that $ G $ is a $ 2 $-group. If $ |\cd(G)|\geqslant 3 $, then $ |\cdc(G)|\geqslant 4 $ by Lemma \ref{nilpotentvalues}(b)(ii). Hence $ |\cd(G)|=2 $. Then it follows that $ G/G' $ is elementary abelian. Let $ M $ be such that $ G'/M $ is a minimal normal subgroup of $ G/M $. Let $ \chi \in \Irr(G/M) $ be non-linear and consider $ \overline{G}=(G/M)/(\ker \chi/M) $. Then $ G'/M \nsubseteq \ker \chi/M $. Note that $ \overline{G}'\subseteq Z(\overline{G}) $ is non-trivial and cyclic. If $ Z(\overline{G}) $ has an element of order $ 4 $, then there exists $ \overline{z}\in Z(\overline{G}) $ such that $ \theta(\overline{z})=\epsilon $ where $ \epsilon $ is a root of unity with $ \epsilon \notin \{-1, 1\} $. Let $ \rho \in \Irr(G)$ be an irreducible constituent of $ \theta $. Then $ \rho (\overline{z})=\rho(1)\theta(\overline{z}) $ and so $ m\epsilon, m\overline{\epsilon} \in \ncv(G) $, where $ m=\rho(1) $, a contradiction. Hence $ |Z(\overline{G})|=2 $.  Note that $ \overline{G}'= Z(\overline{G}) $ and so $ \overline{G} $ is extraspecial. Hence (a) follows.

Suppose $ G $ is not a $ 2 $-group. By Lemma \ref{nilpotentvalues}(b), $ G $ is not nilpotent. Suppose that $ G/G' $ is of exponent $ p $, where $ p\geqslant 4 $. Then $  \epsilon, \epsilon^*, \epsilon^{**}, 0\in \cv(G) $, where $ \epsilon, \epsilon^*, \epsilon^{**} $ are $ p $th roots of unity not equal to $ 1 $. This means that $ |\ncv(G)|\geq 4 $. Hence we may assume that $ G/G' $ is an elementary abelian $ p $-group, $ p\in \{ 2, 3\} $. Suppose that $ G/G' $ has exponent $ 3 $. Then $ \ncv{G}=\{\epsilon, \overline{\epsilon}, 0 \} $, where $ \epsilon $ and $ \overline{\epsilon} $ are cube roots of unity and $ G' $ is abelian. Let $ P $ be a Sylow $ 3 $-subgroup. Suppose that $ P $ is non-abelian. Then $ G/O_{3'}(G) $ is a non-abelian $ 3 $-group and by Lemma \ref{nilpotentvalues}, $ 4 \leqslant \ncv(G/O_{3'}(G))\leqslant \ncv(G) $, a contradiction. Hence $ P $ is elementary abelian. By \cite[Theorem A]{IMNT09}, $ P/O_{3}(G) $ is cyclic of order $ 3 $. If $ O_{3}(G)\not=1 $, then $ G=(O_{3'}(G)\rtimes C_{3})\times O_{3}(G) $ which means that $ m\epsilon\in \cv(G) $, where $ \epsilon \not= 1 $ is a root of unity and $ m\not= 1 $ is a character degree, a contradiction. Hence we may assume that $ O_{3}(G)=1 $. Let $ M $ be a normal subgroup of $ G $ such that $ G'/M $ is a minimal normal subgroup of $ G/M $ and let $ \overline{G}=G/M $. Note that $ \overline{P} $ is of order $ 3 $. Let $ 1\not= \overline{g}\in \overline{G} $ be a $ q $-element, where $ q\not=3 $, is a prime. By Lemma \ref{coprimerootofunity}, $ \overline{g}\notin \ker \chi $ for all non-linear $ \chi\in \Irr(G) $. If $ \overline{g}\in Z(\chi) $, then $ m\epsilon \in \cv(G) $ for some root of unity $ \epsilon\not=1 $, a contradiction. Hence $ \overline{g} $ is a root of unity element and this means that $ \overline{G}/Z(\overline{G})\cong \Gamma _{p} $ with $ p - 1=3 $. Hence $ \overline{G}/Z(\overline{G})\cong A_{12} $. But $ \ncv(A_{12})=4 $, a contradiction.

We may assume that $ G/G' $ is an elementary abelian $ 2 $-group. Let $ P $ be a Sylow $ 2 $-subgroup. By \cite[Theorem A]{IMNT09}, $ P/O_{2}(G) $ is cyclic of order $ 2 $. Consider $ G/O_{2}(G)\cong H\rtimes C_{2} $, where $ H $ is the abelian Hall $ 2' $-subgroup of $ G $. Note that $ G=H\rtimes P $. Let $ T=H\rtimes C_{2} $ and let $ 1\not= c\in C_{2} $. Then $ \chi(c)=0 $ for all non-linear $ \chi \in \Irr(T) $. Hence $ |C_{T}(c)|=\sum_{\chi\in \Irr(T)}|\chi(c)|^{2}= 1^{2} + 1^{2}=2 $. Thus $ C_{T}(c)\subseteq C_{2} $ for all $ 1\not= c\in C_{2} $, that is, $ T $ is a Frobenius group with kernel $ H $ and complement $ C_{2} $. Since $ C_{T}(c) $ is a $ 2 $-group, $ C_{2} $ acts fixed point freely on $ H $. We also have that since $ c $ is an involution, $ c^{-1}hc=h^{-1} $ for all $ h\in H $ by \cite[Proposition 16.9(e)]{Hup98}. This means that $ \langle h \rangle $ is normal in $ T $ for all $ 1\not= h \in H $ and so $ H= \langle h_{1} \rangle\times \langle h_{2} \rangle \times \cdots \times \langle h_{k} \rangle $. We have that $ \langle h_{i} \rangle \rtimes \langle c \rangle$ is a factor group of $ T $ for each $ i\in \{1, 2, \dots, k\} $. It follows that $ \langle h_{i} \rangle \rtimes \langle c \rangle = D_{2t_{i}} $ the dihedral group of order $ 2t_{i} $ with $ t_{i} $ an odd integer. Now the character theory of dihedral groups is well known (see for example \cite{Sot14}). In particular, if $ t_{i}\geqslant 5 $, then $ \ncv(D_{2t_{i}}) > 3 $. Therefore $ t_{i}=3 $ and so $ H $ is an elementary abelian $ 3 $-group. If $ O_{2}(G)=1 $, then by Sakurai's theorem \cite{Sak21}, $ |\cv(G)|=4 $, a contradiction. Hence $ O_{2}(G)\not= 1 $. Finally, it is sufficient to consider when $ P $ is abelian. Then $ P $ is elementary abelian and so we have that $ G=O_{2}(G)\times (C_{3}^*\rtimes C_{2})=C_{2}^*\times (C_{3}^*\rtimes C_{2}) $. This proves (b).
\end{proof}
Theorem \ref{3valuesdl2} establishes Theorem E. We point out that we could not find an example of a group with a non-abelian Sylow $ 2 $-subgroup in Theorem \ref{3valuesdl2}(b).

\section{Finite non-solvable groups with few character values}\label{nonsolvable}
We now turn to character values of non-solvable groups. We start by noting the existence of rational valued extendible characters in simple groups.
\begin{lemma}\cite[Lemma 4.1]{HSTV20}\label{rationalextension}
Let $ S $ be a finite non-abelian simple group. Then there exists a non-principal irreducible character of $ S $ that is extendible to a rational-valued character of $ \Aut(S) $.
\end{lemma}

\begin{theorem}\label{extendiblecharacters}
If $ n\geqslant 5 $, $ n\not= 6 $, then $ A_{n} $ has two non-principal characters of coprime degrees that extend to $ \Aut(A_{n})=S_{n} $. 
\end{theorem}
\begin{proof}
This follows \cite[Theorem 3]{BCLP07} and that $ A_{5} $ has two extendible characters of degrees $ 4 $ and $ 5 $.
\end{proof}
It turns out that a non-rational group will have more character values by Lemma \ref{lem}(d). Hence we want to take advantage of this fact to exclude many groups. The next result tells us exactly which simple groups are composition factors of rational groups. 
\begin{theorem} \label{rationalnonsolvable}\cite[Theorem B and Corollary B.1]{FS89} Let $ S $ be a non-abelian finite simple group. Then
\begin{itemize}
\item[(a)]  $ S $ is a composition factor of a rational group if and only if $ S\in \{ A_{n}$, with $ n\geqslant 5, \PSL_{3}(4), \PSU_{4}(3), \mathrm{PSp}_{4}(3), \mathrm{Sp}_{6}(2), \mathrm{O}_{8}^{+}(2)'  \} $. 
\item[(b)] $ S $ is rational group if and only if $ S\cong \mathrm{Sp}_{6}(2) $ or  $\mathrm{O}_{8}^{+}(2)' $.
\end{itemize}
\end{theorem}

For Theorem \ref{rationalnonsolvable}, we can see that the alternating groups are the only infinite family groups such that they are composition factors of rational groups. We shall consider them in a uniform way. The rest of the groups in the list above can be dealt with using \Atlas{} \cite{CCNPW85} and \GAP {} \cite{GAP}. 

Let $ n $ be a natural number. A partition $ \lambda=(\lambda_{1}, \lambda_{2},\dots, \lambda_{r}) $ of $ n $ such that $ \lambda_{i} $, $ i=1, 2, \dots, r $ are integers, with $ \lambda_{1}\geqslant \lambda_{2} \geqslant \cdots \geqslant \lambda_{r} > 0 $ and $ \lambda_{1} + \lambda_{2} + \dots + \lambda_{r}=n $. The irreducible complex characters of the symmetric group $ S_{n} $ are characterized by partitions and hence can be identified by a corresponding partition. If $ \chi_{\lambda}\in \Irr(S_{n}) $, then an irreducible character of $ A_{n} $ is obtained by restricting $ \chi_{\lambda} $ to $ A_{n} $. In particular, $ \chi_{\lambda} $ is irreducible when restricted to $ A_{n} $ if and only if $ \lambda $ is not self-conjugate. Elements of $ S_{n} $ and $ A_{n} $ are identified by their cycle types. In the next result we show that if $ n\geqslant 15 $, then $ A_{n} $ contains an irreducible character that is extendible to $ S_{n} $ and has four character values that are not positive integers:

\begin{proposition}\label{Alternating4charactervalues}
Let $ S\cong A_{n} $, with $ n\geqslant 15 $. Then $ S $ has an irreducible character $ \chi $ which is extendible to $ S_{n} $ such that $ |\ncv(\chi)|\geqslant 4 $. 
\end{proposition}
\begin{proof}
The right character $ \chi_{\lambda} $ is when $ \lambda=(n-2, 1, 1) $(this is extendible to $ S_{n} $ since it is not self-conjugate). We show in the table below for which elements $ g\in A_{n} $ the character $ \chi_{\lambda} $ take values $ 0, -1, -2, -3 $.

\begin{center}\label{AlternatingTable}
Character values of $ \chi_{\lambda}=(n-2, 1, 1) $ in $ A_{n} $.
\begin{tabular}{c|l|r}
\cline{1-3}
$ n $ & $ g\in A_{n} $ & $ \chi_{\lambda}(g) $ \\
\cline{1-3}
$ n\geqslant 15 $, $ n $ odd  &   $ (n-6, 4, 2) $   &   $ 0 $  \\

  &   $(n-4, 2, 2) $   &   $ -1 $ \\

  &   $ (n-10, 4, 2,2, 2) $   &   $ -2 $  \\

  &   $ (n-8, 2, 2, 2, 2) $   &   $ -3 $\\
\cline{1-3}
$ n\geqslant 16 $, $ n $ even  &   $ (n-2, 2) $   &   $ 0 $  \\

 &   $ (n-8, 4, 2, 2, 2) $   &   $ -1 $ \\

 &   $ (n-6, 2, 2, 2) $   &   $ -2 $ \\

 &   $ (n-12, 4, 2, 2, 2, 2) $   &   $ -3 $ \\
\hline
\end{tabular}
\end{center}
\end{proof}

\begin{theorem}\label{uniqueminimalnormal}
Suppose that $ G $ has a unique non-abelian minimal normal subgroup $ N $. Then either $ |\ncv(G)|\geqslant 4 $ or $ |\ncv(G)|=3 $ and possibly one the following holds:
\begin{itemize}
\item[(a)] $ N=A_{6}\times A_{6} $ and $ G/N $ is a elementary abelian $ 2 $-group.
\item[(b)] $ G\cong S_{5} $.
\end{itemize}
\end{theorem}
\begin{proof}
Suppose that $ N=S_{1}\times S_{2} \times \cdots \times S_{k} $, where $ S_{i}\cong S $ for some non-abelian simple group $ S $ and $ k\geqslant 2 $. By Lemma \ref{rationalextension}, there exists $ \theta \in \Irr(S) $ that is extendible to a rational-valued character of $ \Aut(S) $. Then there exists $ \chi \in \Irr(G) $ that is rational-valued such that $ \chi_{N} =\theta _{1}\times \theta _{2}\times \cdots \times \theta _{k} $. Note that $ \chi_{N} $ is faithful in $ N $ and hence $ |\cv(\chi_{S_{i}})|\geqslant 4 $ for all $ i $ by \cite[Lemma 2]{BCZ95}. Suppose these values are $ \theta_{1}(1), 0, a, b $ with $ a $, a negative rational number by Lemma \ref{lem}(d). Note that $ |a|, |b| < \theta_{1}(1) $. Let $ g_{1}\in S_{1} $, $ h_{i}\in S_{i} $ be such that $ \theta_{1}(g_{1})=a $ and $ \theta_{i}(h_{i})=b $ for $ i=2, \dots, k $. Hence $ \chi(g_{1}\times 1 \times \cdots \times 1)=a\theta(1)^{k-1} \in \ncv(\chi) $. If $ b $ is negative, then $ b\theta(1)^{k-1}\ncv(G) $ and if $ b $ is positive, then $ ab^{k-1}\in \ncv(\chi) $. Since $ 0\in \ncv(\chi) $, then $ |\ncv(G)|\geqslant 3 $. If $ G $ is not a rational group, then by Lemma \ref{lem}(f), $ |\ncv(G)|\geqslant 5 $.

Suppose that $ G $ is a rational group. Then by Theorem \ref{rationalnonsolvable}, we may assume that $ S $ is isomorphic to $ \PSL_{3}(4), \PSU_{4}(3)$, $\mathrm{PSp}_{4}(3), \mathrm{Sp}_{6}(2), \mathrm{O}_{8}^{+}(2)'$ or $ A_{n} $ with $ n\geqslant 5 $. Suppose $ S\ncong A_{6} $, then by the \Atlas{} \cite{CCNPW85} and Theorem \ref{extendiblecharacters}, there exists two extendible characters and using the argument in the first paragraph we get at least four negative character values. If $ S\cong A_{6} $ and $ k\geqslant 3 $, then using the Steinberg character and the argument in the first paragraph, $ 0, -1(1)^{k-1}, -1.9^{k-1}, -1(1)9^{k-2}\in \ncv(G)  $ and so $ |\ncv(G)|\geqslant 4 $. Consider $ k=2 $. By the argument in the first paragraph we have $ 0, -1,-9\in \ncv(G)  $. Then $ \Aut(N)=(\Aut(A_{6}))^{2}\wr C_{2} $. If $ B=(\Aut(A_{6})\times \Aut(A_{6}))\cap G $, then $ G/B\cong C_{2} $. Note that $ G/N $ is a $ 2 $-group. If $ G/N $ is non-abelian, then $ -2\in \ncv(G) $ using Lemma \ref{nilpotentvalues}, a contradiction. It follows that $ G/N $ is elementary abelian $ 2 $-group. 

Hence $ G $ is almost simple group with socle $ S $. By Lemma \ref{rationalextension}, there exists $ \theta \in \Irr(S) $ that is extendible to a rational-valued character of $ \Aut(S) $. Then there exists $ \chi \in \Irr(G) $ that is rational-valued such that $ \chi_{S}\in \Irr(S) $. Note that $ \chi_{S} $ is faithful in $ S $ and hence $ |\cv(\chi_{S})|\geqslant 4 $ by \cite[Lemma 2]{BCZ95}. We may assume that these values are $ \theta(1), 0, a, b $ with $ a $, a negative rational number by Lemma \ref{lem}(d). If $ G $ is not a rational group, then $ G $ has at least two irrational character values $ c, d $ using Lemma \ref{lem}(f). Hence $ |\ncv(G)|\geqslant 4 $.

Suppose that $ G $ is rational group. Note that the only simple groups are $ \mathrm{Sp}_{6}(2)$ and $\mathrm{O}_{8}^{+}(2)' $. Using the \Atlas{} \cite{CCNPW85} and \GAP{} \cite{GAP} for $ \mathrm{Sp}_{6}(2)$, $\mathrm{O}_{8}^{+}(2)' $ and $ S_{n} $ with $ n\leqslant 14 $, it follows that the only group with $ |\ncv(G)|=3 $ is $ S_{5} $. If $ G\cong S_{n} $, where $ n\geqslant 15 $, then the result follows by Proposition \ref{Alternating4charactervalues}.
\end{proof}

\begin{theorem}\label{3negativevalues}
Let $ G $ be a finite group which has no composition factor isomorphic to $ A_{5} $ or $ A_{6} $. If $ |\ncv(G)|\leqslant 3 $, then $ G $ is solvable.
\end{theorem}
\begin{proof}
Suppose the theorem is not true and let $ G $ be a minimal counterexample. If $ N_{1} $ and $ N_{2}$ are two minimal normal subgroups of $ G $, then $ |\ncv(G/N_{i})|\leqslant 3 $ and $ G/N_{i} $ has no composition factor isomorphic to $ A_{5} $ or $ A_{6} $, for $ i=1, 2 $. Hence $ G/N_{i} $ is solvable and so $ G/(N_{1}\cap N_{2})\cong G $ is solvable. 

We may assume that $ G $ has a unique minimal normal subgroup which is non-abelian. Using Theorem \ref{uniqueminimalnormal}, we have our result.
\end{proof}

Theorem B follows immediately from Theorem \ref{3negativevalues} as the following result states:

\begin{corollary}
Let $ G $ be a finite group which has no composition factor isomorphic to $ A_{5} $ or $ A_{6} $. If $ |\cdc(G)|\leqslant 3 $, then $ G $ is solvable.
\end{corollary}

\begin{theorem}\label{NforS5}
Let $ G $ be a finite group with a composition factor isomorphic to $ S\in \{A_{5}, A_{6}\} $.  If $ |\cdc(G)|\leqslant 3 $, then there exists a normal subgroup $ N $ of $ G $ such that either
\begin{itemize}
\item[(a)] $ G/N\cong S_{5} $ and $ |\cd(G)|\geqslant 5 $, or
\item[(b)] $ |G/M|\leqslant 16 $ and $ G/M $ is an elementary abelian $ 2 $-group and $ M/N\cong A_{6}\times A_{6} $ for some normal subgroup $ M $ of $ G $.
\end{itemize}
\end{theorem}
\begin{proof}
Suppose $ G $ has a composition factor isomorphic to $ S\in \{A_{5}, A_{6}\} $. We may assume that $ G $ has a minimal normal subgroup $ N_{1} $ which is a direct product of copies of $ S $. If there is another minimal normal subgroup of $ G $, then let $ N $ be a maximal normal subgroup of $ G $ with respect to $ N\cap N_{1}=1 $. Then $ G/N $ has a unique minimal normal subgroup of $ G $. Using Theorem \ref{uniqueminimalnormal}, the result then follows. If $ \cd(G)|=4 $, then $ 2\in \cdc(G) $ and so $ |\cdc(G)|> 3 $.
\end{proof}

Note that in Theorem \ref{NforS5}(a), $ 2\in \cdc(G/N) $ but however it is not clear if $ 2\in \cdc(G) $ for any $ N $. 
We are now ready to prove Theorem A which we restate below:

\begin{theorem} If $ |\cv(\chi)|\leqslant 4 $ for all non-linear $ \chi \in \Irr(G) $ for a finite group $ G $, then $ G $ is solvable.
\end{theorem}
\begin{proof}
Suppose that the theorem is not true and let $ G $ be a minimal counterexample. If $ G $ has two minimal normal subgroups $ N_{1} $ and $ N_{2} $, then for $ G/N_{1} $ and $ G/N_{2} $, the theorem holds, that is, $ G/N_{1} $ and $ G/N_{2} $ are solvable. This means that $ G $ is solvable, a contradiction. 

We may assume that $ G $ has a unique non-abelian minimal normal subgroup $ N $. Suppose that $ N=S_{1}\times S_{2} \times \cdots \times S_{k} $, where $ S_{i}\cong S $ for some simple group $ S $. By Lemma \ref{rationalextension}, there exists $ \theta \in \Irr(S) $ that is extendible to a rational-valued of $ \Aut(S) $. Then there exists $ \chi \in \Irr(G) $ that is rational-valued such that $ \chi_{N} =\theta _{1}\times \theta _{2}\times \cdots \times \theta _{k} $. Note that $ \chi_{N} $ is faithful in $ N $ and hence $ |\cv(\chi_{S_{i}})|\geqslant 4 $ for all $ i $ by \cite[Lemma 2]{BCZ95}. 

We now show that $ N $ is a simple group. Assume the contrary, that is, $ k\geqslant 2 $. Suppose these values are $ \theta_{1}(1), 0, a, b $, where $ |a|<\chi(1) $ and $ |b|<\chi(1) $. Let $ g_{1}\in S_{1} $, $ h_{i}\in S_{i} $ be such that $ \theta_{1}(g_{1})=a $ and $ \theta_{i}(h_{i})=b $ for $ i=2, \dots, k $. Hence $ \chi(g_{i}\times 1 \times \cdots \times 1)=a\theta(1)^{k-1}  $,  $ \chi(h_{i}\times 1 \times \cdots \times 1)=b\theta(1)^{k-1}  $ and $ \chi(g_{i}\times h_{i} \times \cdots \times h_{i})=ab^{k-1}  $, $ \chi(1)=\theta(1)^{k} $ and $ 0 $ are five distinct values of $ \chi $. Hence we may assume that $ N $ is a simple group and so $ G $ is almost simple. 

Suppose that every non-linear of irreducible character of $ G $ is rational. Then by \cite[Theorem 3.10]{DIM10}, $ G $ is a rational group.  Using the \Atlas{} \cite{CCNPW85} and \GAP{} \cite{GAP} for $ \mathrm{Sp}_{6}(2)$, $\mathrm{O}_{8}^{+}(2)' $ and $ S_{n} $ with $ n\leqslant 14 $ and then Proposition \ref{Alternating4charactervalues} for $ n\geqslant 15 $, we obtain our result.

Hence we may assume that $ G $ has a non-linear irreducible character which is not rational valued. Using \cite[Proposition 7]{BCZ95}, we have that $ N $ is a union of three conjugacy classes of $ G $. Applying Lemma \ref{3conjugacyclasses}, it follows that $ G $ is solvable, our final contradiction.
\end{proof}

\begin{proposition}\label{4charactervaluesexplicit}
Suppose that $ G $ is a non-abelian finite group such $ |\cv(\chi)|= 4 $ for all non-linear $ \chi \in \Irr(G) $. Then one of the following holds:
\begin{itemize}
\item[(a)] $ G $ is a rational group.
\item[(b)] $ G/N $ is a $ 2 $-group for some normal subgroup $ N $.
\item[(c)] $ G/N $ is a Frobenius group with cyclic complement for some normal subgroup $ N $.
\item[(d)] $ G $ has a unique minimal normal subgroup which is an elementary abelian $ p $-group for some odd prime $ p $.
\end{itemize}
\end{proposition}
\begin{proof}
If all the non-linear irreducible characters of $ G $ are rational valued, then we have (a) - (c) by \cite[Theorem 3.12]{DIM10}. Suppose that there is non-linear $ \chi \in \Irr(G) $ that takes irrational values. Then $ G $ has a unique minimal normal subgroup $ N $ which is a union of three conjugacy classes by \cite[Proposition 7]{BCZ95}. Using Lemma \ref{3conjugacyclasses}, it follows that $ N $ is an elementary abelian $ p $-group for some odd prime $ p $.
\end{proof}

\subsection{Examples} We list some examples of groups in Proposition \ref{4charactervaluesexplicit}. A classification seems far fetched. 

\begin{itemize}
\item[(a)] $ G\cong S_{4} $ is an example of a rational group with $ |\cv(\chi)|\leqslant 4 $ for all $ \chi \in \Irr(G) $. We could not find a rational group such that $ |\cv(\chi)|= 4 $ for all non-linear $ \chi \in \Irr(G) $.
\item[(b)] We list examples of groups in Proposition \ref{4charactervaluesexplicit}(b):
\begin{itemize}
\item[(i)] Let $G=\langle \rho,\sigma| \rho^5=\sigma^2,\sigma^{-1}\rho\sigma=\rho^4\rangle=D_{10}$, the dihedral group of order $10$. Then the derived subgroup is $G'=\langle\rho^2\rangle\cong C_5$ and $G/G'\cong C_2$. This is SmallGroup(10,1).
\item[(ii)] $G\cong C_{17}\rtimes C_8$ has a normal subgroup $N\cong D_{34}.C_2$ and $G/N\cong C_2$. This is SmallGroup(136, 12).
\end{itemize}
\item[(c)] The following are examples of groups in Proposition \ref{4charactervaluesexplicit}(c):
\begin{itemize}
\item[(i)] Our first example is the one given in (b)(i).
\item[(ii)] The group $C_5\rtimes D_{10}$, where $D_{10}$ is the dihedral group of order $10$, has $12$ irreducible characters of degree $2$ and each takes on four values. Furthermore, $G$ has nine conjugacy classes of normal subgroups $N$, six of which have quotients $G/N\cong D_{10}.$ It is well known that the dihedral groups of order $2n$ with $n$ odd are Frobenius. Here the kernel is $C_5$ and the complement is $C_2$. This is SmallGroup(50,4). 
\item[(iii)] Let $G\cong C_5^2\rtimes D_{10}$. Then $G$ has a normal subgroup $N$ such that $G/N\cong D_{10}.$ The group $D_{10}$ is Frobenius with complement $C_2$. This is SmallGroup(250,14). 
\end{itemize}
\item[(d)] The following are examples of groups in Proposition \ref{4charactervaluesexplicit}(d). We found the most groups in this case:
\begin{itemize}
\item[(i)] Set $G=C_7\rtimes C_3$. Then the two characters of degree $3$ each assume four values. This is SmallGroup(21,1).
\item[(ii)] Let $G=\rm{He}_3$, the Heisenberg group of order $27$. The group $G$ has a socle $N\cong C_3$. This is SmallGroup(27,3).
\item[(iii)] The group $G=C_9\rtimes C_3$  has two non-linear characters each of degree $3$ and both taking on four values. Note that $G$ has a unique minimal normal subgroup of order $3$ which is its socle. This is SmallGroup(27,4).
\item[(iv)] Set $G=C_3\rtimes S_3.C_2$. The two non-linear characters of $ G $ of degree $4$ assume four  values each and $G$ has a unique minimal normal subgroup $N\cong C_3^2$. This is SmallGroup(36,9).
\item[(v)] Let $G=C_{11}\rtimes C_5$. This is SmallGroup(55,1).
\item[(vi)] Set $G=D_{26}:C_3$. This is SmallGroup(78,1)).
\item[(vii)] Let $G\cong C_2^4\rtimes C_5$. Then $ G $  has socle $C_2^4$ equal to the unique minimimal normal subgroup and all the irreducible non-linear characters assume four values. This is SmallGroup(80,49).
\item[(viii)] The following groups of order $81$ all have unique minimal normal subgroups $N\cong C_3^2$.
    \begin{itemize}
    \item $G\cong C_3^2\rtimes C_9$ (SmallGroup(81,3));
    \item $G\cong C_9\rtimes C_9$ (SmallGroup(81,4));
    \item $G\cong C_3\times \rm{He}_3$ (SmallGroup(81, 12)); and
      \item $G\cong C_3\times C_9\rtimes C_3$ (Smallgroup(81, 13)).
     \end{itemize}
\item[(ix)] The group $G\cong C_7^2\rtimes C_3$  has a unique minimal normal subgroup $N\cong C_7^2$. This is SmallGroup(147,4).
\end{itemize}

\end{itemize}

\section{Problems}
We end the article by stating some problems.
\begin{problem}
Let $ G $ be a solvable group. Is it true that there exist a linear function $ f $ such that 
\begin{center}
$ |\cd(G)|\leqslant f(|\cdc(G)|) $? 
\end{center}

For non-abelian nilpotent groups, we remark that it is easy to show
\begin{center}
 $ |\cd(G)|\leqslant |\cdc(G)| + 1 $
 \end{center}
  using Lemma \ref{nilpotentvalues}(a). For solvable groups, based on our results, we have that 
\begin{itemize}
\item[(a)] If $ |\cd(G)|=1 $, then $ |\cdc(G)|\geqslant 0 $;
\item[(b)] If $ 2 \leqslant|\cd(G)| \leqslant 3 $, then $ |\cdc(G)|\geqslant 2 $;
\item[(c)] If $ |\cd(G)|=4 $, then $ |\cdc(G)|\geqslant 3 $.
\end{itemize}
\end{problem}

The following problems are related in one way or the other.
\begin{problem}
Let $ G $ be a finite group, $ N $ be a normal subgroup of $ G $.
\begin{itemize}
\item[(a)] Is it true that if $ |\cdc(G)|\leqslant 3 $, then $ G $ is solvable?
\item[(b)]  Is it true that $ |\cdc(G/N)|\leqslant |\cdc(G)| $? 
\item[(c)] Describe the structure of $ N $ in the case of Theorem \ref{NforS5}(b) if it exists.
\item[(d)] Classify all non-solvable groups $ G $ such that $ |\cdc(G)|=4 $. Note that the case when $ |\cd(G)|=4 $ was classified in \cite{Mad22}.
\end{itemize}
\end{problem}

\begin{problem}
Classify non-solvable groups such that $ |\cv(\chi)|\leqslant 5 $ for all $ \chi \in \Irr(G) $.
\end{problem}

\begin{problem}
Classify finite groups $ G $ such that $ |\cdc(G)|=2 $.
\end{problem}
We shall consider a counterexample to the fact that the groups in Theorem C are the only ones with the property that $ |\cdc(G)|=2 $.
\begin{remark}
Suppose that $ G $ is a finite group such that $ |\cdc(G)|=2 $. If $ \dl(G)= 4 $, then $ G $ there exists a normal abelian $ 3 $-subgroup $ N $ of $ G $ such that $ G/N\cong S_{4} $ and $ \cd(G)\subseteq \{1, 2, 3, 4, 6, 8, 12\} $. In particular, $ G $ contains an element of order $ 6 $ and there exists $ \chi \in \Irr(G) $ such that either $ \chi(1)=6 $ or $ \chi(1)=12 $.
\end{remark}
\begin{proof}
Consider the abelian subgroup $ G'''\not= 1 $. Then $ \dl(G/G''')=2 $ and $ |\cdc(G/G''')=2 $. It follows that $ G/G'''\cong S_{4} $ by Theorem \ref{derivedlength3}. Using \cite[Corollary 2]{Row}, we have that $ G''' $ is a $ \{2, 3, 5 \} $-group. Let $ T $ be the normal Hall $ 2' $-subgroup of $ G''' $. Then $ \dl(G/T) = 4 $ and $ \chi(1) $ divides $ |G:G'''| $ for all $ \chi \in \Irr(G/T) $ by Ito's theorem. Hence $ \cd(G/T)\subseteq \{1, 2, 3, 4, 6, 8, 12, 24\} $. Let $ \chi(1)\in \{ 6, 12, 24\} $ for some $ \chi \in \Irr(G/T) $. Note that $ \chi $ vanishes on all elements of $ 3 $-singular elements. Consider $ 2 $-elements of $ G/T $. Then $ |\chi(g)|\not= 1  $ by Lemma \ref{prootofunityelement}. Hence $ -1\notin \cv(\chi) $, a contradiction of Lemma \ref{lem}(d). Since $ |\cd(G/T)|\geqslant 5 $ by Theorem \ref{4characterdegrees}, it follows that $ \cd(G/T)=\{ 1, 2, 3, 4, 8\} $. Note that the character degree graph of $ G/T $ is disconnected. Using \cite[Theorem 1]{EA21}, we have that $ G/T\cong S_{4}\times C_{2}^{*} $, a contradiction. This means that $ G''' $ is a $ \{3, 5\} $-group and so $ G $ has a Sylow $ 2 $-subgroup isomorphic to $ D_{8} $. Hence $ G''' $ is a $ 3 $-subgroup by \cite[Theorem 1]{DS21}. 

Since $ G $ satisfies the conditions of $ G/T $ we have that $ \cd(G)\subseteq \{1, 2, 3, 4, 6, 8, 12, 24\} $. Let $ \chi(1)=24 $ for some $ \chi \in \Irr(G) $. Then all $ \chi $ vanishes on all $ 2 $-singular elements. Suppose that $ g $ is a $ 3 $-element. Then $ |\chi(g)|\not= 1  $ by Lemma \ref{prootofunityelement}. Hence $ -1\notin \cv(\chi) $, a contradiction of Lemma \ref{lem}(d). Hence $ 24\notin \cd(G) $. 

If $ G $ has no character of degree $ 6 $ or $ 12 $, then the character degree graph is disconnected and by the argument in the first paragraph we reach a contradiction. Suppose that $ G $ does not have elements of order $ 6 $. It is sufficient to show that $ G $ has no character of degree $ 6 $ or $ 12 $. Suppose there is a $ \chi\in \Irr(G) $ such that $ \chi(1)\in \{6,12\} $. If $ g $ is a $ 2 $-element or a $ 3 $-element, then by Lemma \ref{prootofunityelement}, $ -1\notin \cv(\chi) $, our final contradiction.
\end{proof}

\section*{Acknowledgements}
The first author was supported by the National Research Foundation (NRF) of South Africa under Grant Number 150857. All opinions, findings and conclusions or recommendations expressed in this publication are those of the authors and the NRF does not accept any liability in regard thereto. Part of this work was done whilst the first author was visiting University of Zambia. He thanks the Department of Mathematics and Statistics for their hospitality.

 The work of the third author is supported by the Mathematical Center in Akademgorodok under the agreement No. 075-15-2022-281 with the Ministry of Science and Higher Education of the Russian Federation.


\begin{thebibliography}{1}
\bibitem{BCZ95} Y. Berkovich, D. Chillag and E. Zhmud, Finite groups in which all non-linear irreducible characters have three character values, \emph{Houston J. Math} \textbf{21} (1995) 17--28.
\bibitem{BCLP07} M. Bianchi, D. Chillag, M. Lewis and E. Pacifici, Character degree graphs that are complete graphs, \emph{Proc. Amer. Math. Soc.} \textbf{135} (2007) 671--676.
\bibitem{CCNPW85} J. H. Conway, R. T. Curtis, S. P. Norton, R. A. Parker and R. A. Wilson, \textit{Atlas of Finite Groups}, Oxford, Clarendon Press, 1985.
\bibitem{DIM10} M. R. Darafsheh, A. Iranmanesh and S. A. Moosavi, Groups whose non-linear irreducible characters are rational valued, \emph{Arch. Math} \textbf{94} (2010) 411--418.
\bibitem{DS21} M. R. Darafsheh, H. Sharifi, Sylow $2$-subgroups of solvable $\mathbb{Q} $-groups, \emph{Extracta Math.} \textbf{22} (2007), no. 1, 83-–91.
\bibitem{DGD24} S. C. Debon, D. Garcia-Lucas and A. del Rio, The Gruenberg-Kegel graph of finite solvable rational groups, \emph{J. Algebra} \textbf{642} (2024) 470--479.
\bibitem{DL15} N. Du, M. L. Lewis, Groups with four character degrees and derived length four, \emph{Comm. Algebra} \textbf{43} (2015) 4660--4673.
\bibitem{EA21} T. Erkoç, G. Akar, Rational groups whose character degree graphs are disconnected, \emph{C. R. Math. Acad. Sci. Paris} \textbf{360} (2022), 711--715.
\bibitem{FS89} W. Feit and G. M. Seitz, On finite rational groups and related topics, \emph{Illinois J. Math.} \textbf{33} (1989) 103--131.
\bibitem{GAP} The GAP. GAP - Groups, Algorithms and Programming, Version 4.8.7. http://www.gap-system.org, 2017.
\bibitem{HSTV20} N. N. Hung, A. A. Schaeffer Fry, H. P. Tong-Viet and C. R. Vinroot, On the number of irreducible real-valued characters of a finite group, \emph{J. Algebra} \textbf{555} (2020) 275--288.
\bibitem{Hup98} B. Huppert. \emph{Character theory of finite groups}, Volume 25. Walter de Gruyter, 1998.
\bibitem{Isa06} I. M. Isaacs, \textit{Character Theory of Finite Groups}, Amer. Math. Soc., Providence, Rhode Island, 2006.
\bibitem{Isa08} I. M. Isaacs, \textit{Finite Group Theory}, 
Amer. Math. Soc., Providence, Rhode Island, 2008.
\bibitem{IK98} I. M. Isaacs, G. Knutson, Irreducible character degrees and normal subgroups, \emph{J. Algebra} \textbf{199} (1998) 302--326.
\bibitem{IMNT09} I. M. Isaacs, A. Moretó, G. Navarro, and P. H. Tiep, Groups with just one character degree divisible by a given prime, \emph{Trans. Amer. Math. Soc.} \textbf{361} (2009) 6521--6547.
\bibitem{LMT23} M. L. Lewis, L. Morotti and H. P. Tong-Viet, Group elements whose character values are roots of unity, \emph{Vietnam J. Math.}, \textbf{52} (2024) 379--388.
\bibitem{Mad22} S. Y. Madanha, Finite groups with few character values, \emph{Comm. Algebra} \textbf{50} (2022) 308--312.
\bibitem{QSY04} G. Qian, W. Shi and X. You, Conjugacy classes outside a normal subgroup, \emph{Comm. Algebra} \textbf{32} (2004) 4809--4820.
\bibitem{QS04} G. Qian and W. Shi, A note on character degrees of finite groups, \emph{J. Group Theory} \textbf{7} (2004) 187--196.
\bibitem{Rie00} J. M. Riedl, Fitting height of solvable groups with few character degrees, \emph{J. Algebra} \textbf{233} (2000) 287--308.
\bibitem{Row} R. Row, Groups whose characters are rational-valued, \emph{J. Algebra} \textbf{40} (1976) 280--299.
\bibitem{Sak21} T. Sakurai, Finite groups with very few character values, \emph{Comm. Algebra} \textbf{49} (2021) 658--661.
\bibitem{SS98} M. Shahryari and M. A. Shahabi, Subgroups which are the union of three conjugate classes, \emph{J. Algebra} \textbf{207} (1998) 326--332.
\bibitem{Sot14} M. M. Soto, \emph{Irreducible representations of $ D_{2n} $}, Masters' Thesis, California State University, 2014
\end{thebibliography}
\end{document}